\documentclass[11pt]{amsart}

\usepackage{graphicx,amscd,amsmath,amsfonts,amssymb} 
\usepackage{amssymb}
\usepackage{amsmath}
\usepackage{amsfonts}
\usepackage{graphicx}
\usepackage[initials]{amsrefs}

\newtheorem{theorem}{Theorem}[section]

\newtheorem{corollary}[theorem]{Corollary}

\newtheorem{remark}[theorem]{Remark}
\newtheorem{remarks}[theorem]{Remarks}
\newtheorem{lemma}[theorem]{Lemma}

\numberwithin{equation}{section}

\newcommand{\N}{\mathbb{N}} 
\newcommand{\R}{\mathbb{R}} 
\newcommand{\Q}{\mathbb{Q}} 
\newcommand{\ve}{\varepsilon}
\newcommand{\al}{\alpha}
\newcommand{\la}{\lambda}

\begin{document}

\title[Lineability in sequence and function spaces]{Lineability in sequence and function spaces}

\author[Ara\'{u}jo]{G. Ara\'{u}jo\textsuperscript{*}}
\address[G. Ara\'{u}jo]{Departamento de Matem\'{a}tica \\
Universidade Federal da Para\'{\i}ba \\
Jo\~{a}o Pessoa - PB \\
58.051-900 (Brazil)}
\email{gdasaraujo@gmail.com}

\author[Bernal]{L. Bernal-Gonz\'{a}lez\textsuperscript{**}}
\address[L. Bernal-Gonz\'{a}lez]{Departamento de An\'{a}lisis Matem\'{a}tico \\
Universidad de Sevilla \\
Apdo. 1160 \\
Avenida Reina Mercedes \\
Sevilla, 41080 (Spain).}
\email{lbernal@us.es}

\author[Mu\~{n}oz]{G.A. Mu\~{n}oz-Fern\'{a}ndez\textsuperscript{***}}
\address[G. A. Mu\~{n}oz-Fern\'{a}ndez]{Departamento de An\'{a}lisis Matem\'{a}tico \\
Facultad de Ciencias Matem\'{a}ticas \\
Plaza de Ciencias 3 \\
Universidad Complutense de Madrid \\
Madrid, 28040 (Spain).}
\email{gustavo\_fernandez@mat.ucm.es}

\author[Prado]{J.A. Prado-Bassas\textsuperscript{**}}
\address[J. A. Prado-Bassas]{Departamento de An\'{a}lisis Matem\'{a}tico \\
Universidad de Sevilla \\
Apdo. 1160 \\
Avenida Reina Mercedes \\
Sevilla, 41080 (Spain).}
\email{bassas@us.es}

\author[Seoane]{J.B. Seoane-Sep\'{u}lveda\textsuperscript{***}}
\address[J. B. Seoane-Sep\'{u}lveda]{ICMAT and Departamento de An\'{a}lisis Matem\'{a}tico \\
Facultad de Ciencias Matem\'{a}ticas \\
Plaza de Ciencias 3 \\
Universidad Complutense de Madrid \\
Madrid, 28040 (Spain).}
\email{jseoane@ucm.es}

\thanks{\textsuperscript{*}Supported by PDSE/CAPES 8015/14-7.}
\thanks{\textsuperscript{**}Supported by the Plan Andaluz de Investigaci\'{o}n de la Junta de Andaluc\'{\i}a FQM-127 Grant P08-FQM-03543 and by MEC Grant MTM2012-34847-C02-01.}
\thanks{\textsuperscript{***}Supported by the Spanish Ministry of Science and Innovation, grant MTM2012-34341.}

\subjclass[2010]{Primary 28A20; Secondary 15A03, 26A24, 40A05, 46A45.}
\keywords{Dense-lineability, algebrability, measurable function, differentiable function, separate continuity, divergent series.}

\begin{abstract}
It is proved the existence of large algebraic structures \break --including large vector subspaces or infinitely generated free algebras-- inside, among others, the family of Lebesgue measurable functions that are surjective in a strong sense, the family of nonconstant differentiable real functions vanishing on dense sets, and the family of non-continuous separately continuous real functions. Lineability in special spaces of sequences is also investigated. Some of our findings complete or extend a number of results by several authors.
\end{abstract}

\maketitle

\section{Introduction and notation}

\quad Lebesgue (\cites{lebesgue1904}, 1904) was probably
the first to show an example of a real function on the reals satisfying the rather surprising
property that it takes on each real value in any nonempty open set (see also \cites{gelbaumolmsted1964,gelbaumolmsted2003}).
The functions satisfying this property are called
\emph{everywhere surjective} (functions with even more stringent properties can be found in \cite{foran1991,jordan1998}).
Of course, such functions are nowhere continuous but, as we will see later, it is possible to construct a \emph{Lebesgue measurable}
everywhere surjective function. Entering a very different realm, in 1906 Pompeiu \cite{pompeiu1906} was able
to construct a nonconstant differentiable function on the reals whose derivative \emph{vanishes on a dense set.}
Passing to several variables, the first problem one meets related to the ``minimal regularity'' of functions at a elementary level is that of whether separate continuity implies continuity, the answer being given in the negative.
In this paper, we will consider the families consisting of each of these kinds of functions, as well as two special families of sequences, and analyze the existence of large algebraic structures inside all these families. Nowadays the topic of lineability has had a major influence in many different areas on mathematics, from Real and Complex Analysis \cite{israel}, to Set Theory \cite{gamezseoane2013}, Operator Theory \cite{hernandezruizsanchez2015}, and even (more recently) in Probability Theory \cite{fenoyseoane2015}. Our main goal here is to continue with this ongoing research.

\vskip .15cm

Let us now fix some notation. As usual, we denote by \,$\N, \, \Q$ \,and \,$\R$ \,the set of positive integers, the set of rational numbers and the set of all real numbers, respectively.
The symbol \,${\mathcal C} (I)$ \,will stand for the vector space of all real continuous functions defined on an interval \,$I \subset \R$.
In the special case \,$I = \R$, the space \,${\mathcal C} (\R )$ \,will be endowed with the topology of the convergence in compacta.
It is well known that  \,${\mathcal C} (\R )$ \,under this topology is an $F$-space, that is, a complete metrizable topological vector space. 

\vskip .15cm

By \,$\mathcal{MES}$ \,it is denoted the family of Lebesgue measurable everywhere surjective functions \,$\R \to \R$. A function \,$f:\R \to \R$ \,is said to be a
\emph{Pompeiu function} (see Figure \ref{Pompeiu}) provided that it is differentiable and \,$f'$ \,vanishes on a dense set in \,$\R$. The symbols \,${\mathcal P}$ \,and \,$\mathcal{DP}$ \,stand
for the vector spaces of Pompeiu functions and of the derivatives of Pompeiu functions, respectively. Additional notation will be rather usual and, when needed, definitions will be provided.

\begin{figure}
	\centering
	\includegraphics[width=0.8\textwidth]{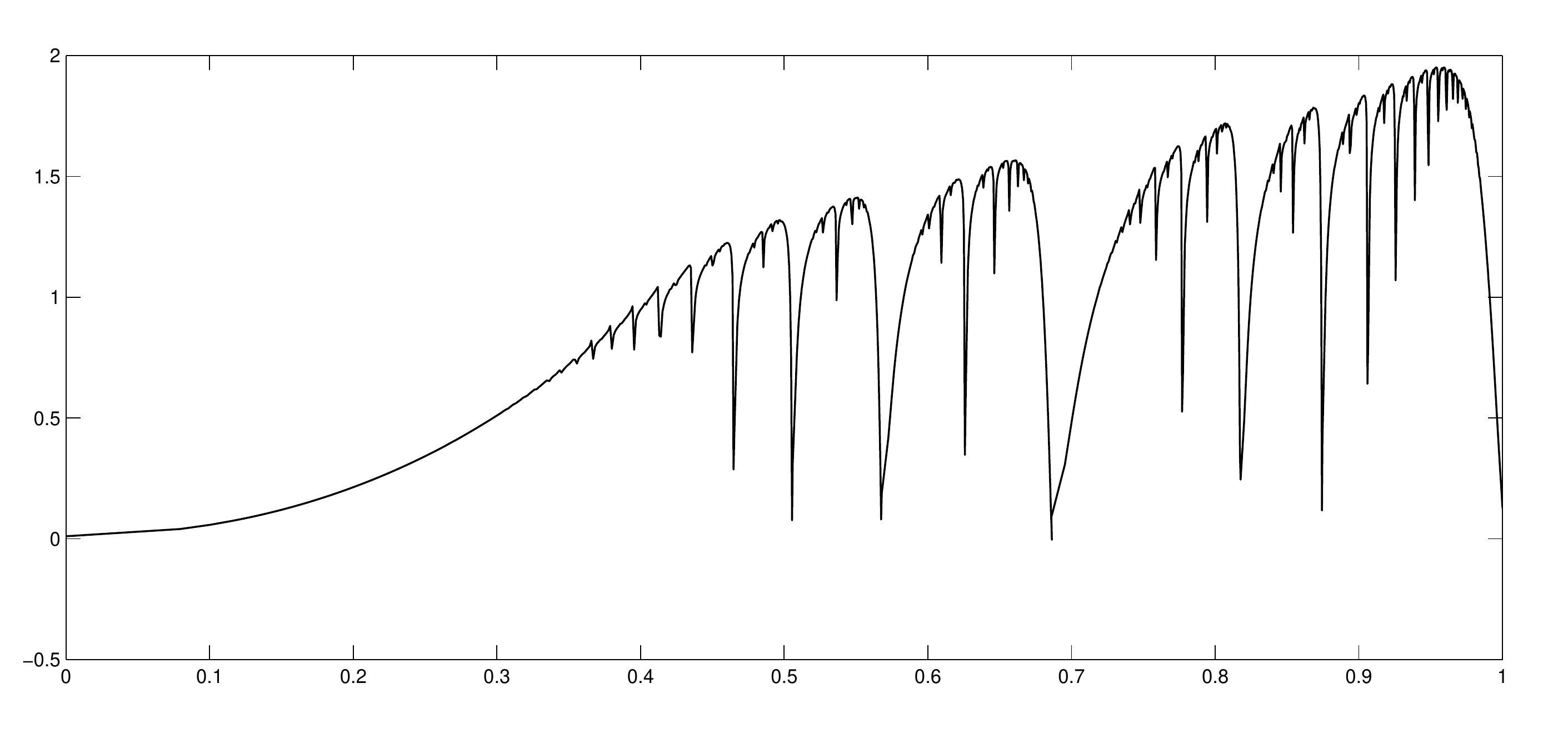}
	\caption{Rough sketch of the graph of Pompeiu's original example.}\label{Pompeiu}
\end{figure}

\vskip .15cm

The organization of this paper is as follows. In section 2, a number of concepts concerning the linear or algebraic structure of sets inside a vector space or a linear algebra, together with some examples related to everywhere surjectivity and special derivatives, will be recalled. Sections 3, 4, and 5 will focus on diverse lineability properties of the families \,$\mathcal{MES}$, $\mathcal{P}$,  \,$\mathcal{DP}$, and certain subsets of discontinuous functions, so completing or extending a number of known results about several strange classes of real functions. Concerning sequence spaces, section 6 will deal with subsets of convergent and divergent series for which classical tests of convergence fail and, finally, in section 7 convergence in measure versus convergence almost everywhere will be analyzed in the space of sequences of measurable Lebesgue functions on the unit interval.

\section{Lineability notions}

A number of concepts have been coined in order to describe the
algebraic size of a given set; see  \cite{arongurariyseoane2005,bayart2005,bernal2010,enflogurariyseoane2014,gurariyquarta2004}
(see also the survey paper \cite{bernalpellegrinoseoane2014} and the forthcoming book \cite{aronbernalpellegrinoseoane2015}
for an account of lineability properties of specific subsets of vector spaces).
Namely, if \,$X$ \,is a vector space, $\alpha$ \,is a cardinal number and \,$A \subset X$, then \,$A$ \,is said to be:
\begin{enumerate}
\item[$\bullet$] {\it lineable} if there is an infinite dimensional vector space $M$ such that $M \setminus \{0\} \subset A$,
\item[$\bullet$] {\it $\alpha$-lineable} if there exists a vector space $M$ with dim$(M) = \alpha$ and $M \setminus \{0\} \subset A$
(hence lineability means $\aleph_0$-lineability,
where $\aleph_0 = {\rm card}\,(\N )$, the cardinality of \,$\N$), and
\item[$\bullet$] {\it maximal lineable} in $X$ if $A$ is ${\rm dim}\,(X)$-lineable.
\end{enumerate}
If, in addition, $X$ is a topological vector space, then $A$ is said to be:
\begin{enumerate}
\item[$\bullet$] {\it dense-lineable} in \,$X$
\,whenever there is a dense vector subspace $M$ of $X$ satisfying \,$M \setminus \{0\} \subset A$
\,(hence dense-lineability implies lineability as soon as dim$(X) = \infty$), and
\item[$\bullet$] {\it maximal dense-lineable} in $X$
whenever there is a dense vector subspace $M$ of $X$ satisfying $M \setminus \{0\} \subset A$ and
dim$\,(M) =$ dim$\,(X)$.
\end{enumerate}
And, according to \cite{aronperezseoane2006,bartoszewiczglab2013}, when \,$X$ \,is a topological vector space contained in some (linear) algebra then
$A$ is called:
\begin{enumerate}
\item[$\bullet$] {\it algebrable} if there is an algebra \,$M$ so
    that $M \setminus \{0\} \subset A$ and $M$ is infinitely generated, that is, the cardinality of any system of generators of \,$M$ is infinite.
\item[$\bullet$] {\it densely algebrable} in $X$ if, in addition, $M$ can be taken dense in $X$.
\item[$\bullet$] {\it $\alpha$-algebrable} if there is an $\alpha$-generated algebra \,$M$ with \,$M \setminus \{0\} \subset A$.
\item[$\bullet$] {\it strongly $\alpha$-algebrable} if there exists an $\alpha$-generated {\it free} algebra \,$M$ with \,$M \setminus \{0\} \subset A$ (for $\alpha = \aleph_0$, we simply say {\it strongly algebrable}).
\item[$\bullet$] {\it densely strongly $\alpha$-algebrable} if, in addition, the free algebra \,$M$ can be taken dense in \,$X$.
\end{enumerate}

\vskip .15cm

Note that if $X$ is contained in a commutative algebra then a set $B \subset X$ is a generating set of some free algebra contained in $A$
if and only if for any $N \in \N$, any nonzero polynomial $P$ in $N$ variables without constant term and any distinct $f_1,...,f_N \in B$, we have
$P(f_1, \dots ,f_N) \ne 0$ and $P(f_1, \dots ,f_N) \in A$. Observe that strong $\alpha$-algebrability $\Longrightarrow$ $\alpha$-algebrability $\Longrightarrow$ $\alpha$-lineability, and none of these implications can be reversed; see \cite[p.~74]{bernalpellegrinoseoane2014}.

\vskip .15cm

In \cite{arongurariyseoane2005}
the authors proved that the set of \emph{everywhere surjective} functions \,$\R \to \R$ \,is $2^{\mathfrak c}$-lineable,
which is the {best} possible result in terms of dimension (we have denoted by \,$\mathfrak{c}$ \,the cardinality of the continuum). In other words, the last set is maximal lineable in the space of all real functions.
Other results establishing the degree of lineability of more stringent classes of functions can be found in \cite{bernalpellegrinoseoane2014} and the references contained in it.

\vskip .15cm

Turning to the setting of more regular functions, in \cite{gamezmunozsanchezseoane2010} the following results are proved: the set of \emph{differentiable} functions
on \,$\R$ \,whose derivatives are discontinuous
almost everywhere is $\mathfrak{c}$-lineable; given a non-void compact interval \,$I \subset \R$, the family of differentiable
functions whose derivatives are discontinuous almost everywhere on \,$I$ \,is dense-lineable in the space \,${\mathcal C} (I)$, endowed with
the supremum norm; and the class of differentiable functions on \,$\R$ \,that are monotone on no interval is $\mathfrak{c}$-lineable.

\vskip .15cm

Finally, recall that every bounded variation function on an interval \,$I \subset \R$ \,(that is, a function satisfying
$\sup \{\sum_{i=1}^n |f(t_i) - f(t_{i-1})|: \, \{t_1 < t_2 < \cdots < t_n\} \subset I, \, n \in \N \} < \infty$) is {\it differentiable almost everywhere.}
A continuous bounded variation function \,$f:I \to \R$ \,is called strongly singular whenever \,$f'(x) = 0$ \,for almost every \,$x \in I$ \,and, in addition, \,$f$ \,is nonconstant on any subinterval of \,$I$. Balcerzak {\it et al.}~\cite{balcerzakbartoszewiczfilipczak2013} showed that the set of strongly singular functions on \,$[0,1]$ \,is densely strongly $\mathfrak{c}$-algebrable in \,${\mathcal C} ([0, 1])$.

\vskip .15cm

A number of results related to the above ones will be shown in the next two sections.

\section{Measurable functions}

\quad Our aim in this section is to study the lineability of the family of Lebesgue measurable functions \,$\R \to \R$ \,that are everywhere surjective, denoted $\mathcal{MES}$. This result is quite surprising, since (as we can see in \cite{gamez2011,gamezmunozsanchezseoane2010}), the class of everywhere surjective functions contains a $2^{\mathfrak{c}}$-lineable set of non-measurable ones (called {\em Jones functions}).

\begin{theorem}\label{Thm-MES-c-lineable}
The set \,$\mathcal{MES}$ \,is $\mathfrak{c}$-lineable.
\end{theorem}

\begin{proof}
Firstly, we consider the everywhere surjective function furnished in \cite{gamezmunozsanchezseoane2010}*{Example 2.2}.
For the sake of convenience, we reproduce here its construction.
Let $(I_n)_{n\in\mathbb N}$ be the collection of all open intervals with rational endpoints. The interval
$I_1$ contains a Cantor type set, call it $C_1$. Now, $I_2 \setminus C_1$ also contains a Cantor type
set, call it $C_2$. Next, $I_3\setminus (C_1 \cup C_2)$ contains, as well, a Cantor type set, $C_3$.
Inductively, we construct a family of pairwise disjoint Cantor type sets, $(C_n)_{n\in\mathbb N}$, such that for every
$n \in \mathbb{N}$, $I_n \setminus (\bigcup_{k=1}^{n-1} C_k) \supset C_n$. Now,
for every $n \in \N$, take any bijection $\phi_n : C_n \to \R$,
and define $f\colon \mathbb{R} \to \mathbb{R}$ as
    \[f(x)=\begin{cases}
    \phi_n(x)&\text{if }x\in C_n,\\
    0&\text{otherwise.}
    \end{cases}\]
Then $f$ is clearly everywhere surjective.
Indeed, let $I$ be any~interval in $\mathbb{R}$. There exists $k \in \mathbb{N}$ such that $I_k \subset I$.
Thus \,$f(I) \supset f(I_k) \supset f(C_k) = \phi_k(C_k) = \R$.

\vskip .15cm

But the novelty of the last function is that \,$f$ \,is, in addition, zero almost everywhere, and in particular, it is (Lebesgue) {\it measurable.}
That is, $f \in \mathcal{MES}$.

\vskip .15cm

Now, taking advantage of the approach of \cite{arongurariyseoane2005}*{Proposition 4.2}, we are going to construct a vector space that shall be useful later on. Let
$$
\Lambda := {\rm span} \, \{\varphi_\al : \, \al > 0\},
$$
where \,$\varphi_\al (x) := e^{\al x} - e^{- \al x}$. Then \,$M$ \,is a $\mathfrak c$-dimensional vector space because the functions
\,$\varphi_\al$ $(\al > 0)$ \,are linearly independent. Indeed, assume that there are scalars \,$c_1, \dots ,c_p$ \,(not all \,$0$) as well as positive reals
$\al_1, \dots ,\al_p$ \,such that \,$c_1 \varphi_{\al_1} (x) + \cdots + c_p \varphi_{\al_p} (x) = 0$ \,for all \,$x \in \R$.
Without loss of generality, we may assume that \,$p \ge 2$, $c_p \ne 0$ \,and \,$\al_1 < \al_2 < \cdots < \al_p$. Then
\,$\displaystyle\lim_{x \to +\infty} (c_1 \varphi_{\al_1} (x) + \cdots + c_p \varphi_{\al_p} (x)) = +\infty$ \,or \,$-\infty$, which is clearly a contradiction.
Therefore \,$c_1 = \cdots = c_p = 0$ \,and we are done. Note that each nonzero member
\,$g = \sum_{i=1}^p c_i \, \varphi_{\al_i}$ (with the \,$c_i$'s and the \,$\al_i$'s as before) \,of \,$\Lambda$ \,is (continuous and) surjective because
\,$\lim_{x \to +\infty} g(x) = +\infty$ \,and \,$\lim_{x \to -\infty} g(x) = -\infty$ \,if \,$c_p > 0$ (with the values of the limits interchanged if \,$c_p < 0$).

\vskip .15cm

Next, we define the vector space
$$
M := \{g \circ f: \, g \in \Lambda\}.
$$
Observe that, since the \,$f$ \,is measurable and the functions \,$g$ \,in \,$\Lambda$ \,are continuous, the members of \,$M$ \,are measurable.
Fix any \,$h \in M \setminus \{0\}$. Then, again, there are finitely many scalars $c_1, \dots ,c_p$ \,with \,$c_p \ne 0$, and positive reals
$\al_1 < \al_2 < \cdots < \al_p$ \,such that \,$g = c_1 \varphi_{\al_1} + \cdots + c_p \varphi_{\al_p}$ \,and \,$h = g \circ f$. Now, fix a non-degenerate
interval \,$J \subset \R$. Then \,$h(J) = g(f(J)) = g(\R ) = \R$, which shows that \,$h$ \,is everywhere surjective. Hence \,$M \setminus \{0\} \subset \mathcal{MES}$.

\vskip .15cm

Finally, by using the linear independence of the functions \,$\varphi_\al$ \,and the fact that \,$f$ \,is surjective, it is easy to see that the functions \,$\varphi_\al \circ f$ $(\al > 0)$ \,are linearly independent, which entails that \,$M$ \,has dimension \,$\mathfrak c$, as required.
\end{proof}

In \cite[Example 2.34]{wisehall1993} it is exhibited one sequence of measurable everywhere surjective functions tending pointwise to zero. With Theorem \ref{Thm-MES-c-lineable} in hand,
we now get a plethora of such sequences, and even in a much easier way than \cite{wisehall1993}.

\begin{corollary}
The family of sequences \,$\{f_n\}_{n \ge 1}$ \,of Lebesgue measurable functions \,$\R \to \R$ \,such that \,$f_n$ \,converges pointwise to zero and such that
 \,$f_n(I) = \R$ \,for any positive integer \,$n$ \,and each non-degenerate interval $I$, is $\mathfrak{c}$-lineable.
\end{corollary}

\begin{proof}
Consider the family \,$\widetilde{M}$ \,consisting of all sequences \,$\{h_n\}_{n \ge 1}$ \, given by \,$h_n (x) = h(x)/n$ \,where the functions \,$h$ \,run over the vector space \,$M$ \,constructed in the last theorem. It is easy to see that \,$\widetilde{M}$ \,is a $\mathfrak{c}$-dimensional vector subspace of \,$(\R^{\R})^{\N}$, that each \,$h_n$ \,is measurable, that \,$h_n(x) \to 0$ $(n \to \infty )$ \,for every \,$x \in \R$, and that every \,$h_n$ \,is everywhere surjective if \,$h$ \,is not the zero function.
\end{proof}

\begin{remark}
{\rm It would be interesting to know whether \,$\mathcal{MES}$ \,is --likewise the set of everywhere surjective functions-- maximal lineable in \,$\R^{\R}$ (that is, $2^{\mathfrak c}$-lineable).}
\end{remark}

\section{Special differentiable functions}

\quad In this section, we analyze the lineability of the set of Pompeiu functions that are not constant on any interval. Of course, this set is not a vector space.

\vskip .15cm

Firstly, the following version of the well-known Stone--Weierstrass density theorem (see e.g.~\cite{rudin1991}) for the space \,$\mathcal{C}(\R)$ \,will be relevant to the proof of our main result. Its proof is a simple application of the original Stone--Weierstrass theorem for \,$\mathcal{C}(S)$ \,(the Banach space of continuous functions \,$S \to \R$, endowed with the uniform distance, where \,$S$ \,is a compact topological space) together with the fact that convergence in \,$\mathcal{C}(\R)$ \,means
convergence on each compact subset of \,$\R$. So we omit the proof.

\begin{lemma}\label{SWforC(R)}
Suppose that \,$\mathcal A$ \,is a subalgebra of \,${\mathcal C} (\R )$ satisfying the following properties:
\begin{enumerate}
\item [\rm (a)] Given \,$x_0 \in \R$ \,there is $F \in {\mathcal A}$ \,with \,$F(x_0) \ne 0$.
\item [\rm (b)] Given a pair of distinct points \,$x_0,x_1 \in \R$, there exists \,$F \in {\mathcal A}$ \,such that \,$F(x_0) \ne F(x_1)$.
\end{enumerate}
Then \,$\mathcal A$ \,is dense in \,${\mathcal C} (\R )$.
\end{lemma}

In \cite[Proposition 7]{balcerzakbartoszewiczfilipczak2013}, Balcerzak, Bartoszewicz and Filipczak established a nice algebrability result by using the so-called
\textit{exponential-like functions,} that is, the functions \,$\varphi:\R \to \R$ \, of the form
$$
\varphi (x) = \sum_{j=1}^{m} a_j e^{b_j x}
$$
for some \,$m \in \N$, some $a_1,\ldots,a_m\in\R\setminus\{0\}$ \,and some distinct \,$b_1,\ldots,b_m\in \R\setminus\{0\}$.
By \,${\mathcal E}$ we denote the class of exponential-like functions.
The following lemma (see \cite{bernal2014} or \cite[Chapter 7]{aronbernalpellegrinoseoane2015}) is a slight variant of the mentioned Proposition 7
of \cite{balcerzakbartoszewiczfilipczak2013}.

\begin{lemma} \label{Lemma-algebrabilitycriterium}
Let \,$\Omega$ \,be a nonempty set and \,${\mathcal F}$ \,be a family of functions \,$\Omega \to \R$.
Assume that there exists a function \,$f \in {\mathcal F}$
such that \,$f(\Omega )$ is uncountable and \,$\varphi \circ f \in {\mathcal F}$ \,for every \,$\varphi \in {\mathcal E}$.
Then \,${\mathcal F}$ is strongly $\mathfrak{c}$-algebrable. More precisely, if \,$H \subset (0,\infty )$ \,is a set with
\,{\rm card}$(H) = \mathfrak{c}$ \,and linearly independent over the field \,$\Q$, then
$$
\{\exp \circ \, (rf): \, r \in H\}
$$
is a free system of generators of an algebra contained in \,${\mathcal F} \cup \{0\}$.
\end{lemma}

Lemma \ref{Lemma-denselystralgebrable} below is an adaptation of a result that is implicitly contained in \cite[Section 6]{bartoszewiczbieniasfilipczakglab2014}.
We sketch the proof for the sake of completeness.

\begin{lemma}\label{Lemma-denselystralgebrable}
Let \,$\mathcal{F}$ \,be a family of functions in \,$\mathcal{C}(\R)$. Assume that there exists a strictly monotone function \,$f\in \mathcal{F}$ \,such that \,$\varphi\circ f\in \mathcal{F}$ \,for every exponential-like function \,$\varphi$. Then \,$\mathcal{F}$ \,is densely strongly $\mathfrak{c}$-algebrable in \,$\mathcal{C}(\R)$.
\end{lemma}

\begin{proof}
If \,$\Omega = \R$ \,then \,$f(\Omega )$ \,is a non-degenerate interval, so it is an uncountable set.
Then, it is sufficient to show that the algebra \,$\mathcal A$ \,generated by the system
\,$\{\exp \circ \, (rf): \, r \in H\}$
\,given in Lemma \ref{Lemma-algebrabilitycriterium} is dense. For this, we invoke Lemma \ref{SWforC(R)}.
Take any \,$\alpha \in H \subset (0,+\infty )$. Given \,$x_0 \in \R$, the function \,$F(x) := e^{\alpha \, f(x)}$ \,belongs to
\,$\mathcal A$ \,and satisfies \,$F(x_0) \ne 0$. Moreover, for prescribed distinct points \,$x_0,x_1 \in \R$, the same function \,$F$ \,fulfills
\,$F(x_0) \ne F(x_1)$, because both functions \,$f$ \,and \,$x \mapsto e^{\al \, x}$ \,are one-to-one.
As a conclusion, $\mathcal A$ \,is dense in \,${\mathcal C} (\R )$.
\end{proof}

Now we state and prove the main result of this section.

\begin{theorem}\label{Thm-Pnonconstant-algebrable}
The set of functions in \,$\mathcal P$ \,that are nonconstant on any non-degenerated interval of \,$\R$ \,is densely strongly $\mathfrak{c}$-algebrable in \,${\mathcal C} (\R )$.
\end{theorem}

\begin{proof}
From \cite[Example 3.11]{wisehall1993} (see also \cite[Example 13.3]{vanrooijschikhof1982}) we know that there exists a derivable {\it strictly increasing} real-valued function \,$(a,b)\to (0,1)$ \, (with \,$f((a,b)) = (0,1)$) \,whose derivative vanishes on a dense set and yet does not vanish everywhere. By composition with the function \,$x \mapsto {b-a \over \pi} \arctan x + {a + b \over 2}$,
we get a strictly monotone function \,$f : \R \to \R$ \,satisfying that
$$D := \{x\in\R \,: \, f'(x)=0\}$$
is dense in \,$\R$ \,but \,$D \ne \R$. Observe that, in particular, $f$ \,is a Pompeiu function that is nonconstant on any interval.

\vskip .15cm

According to Lemma \ref{Lemma-denselystralgebrable}, our only task is to prove that, given a prescribed function
\,$\varphi \in {\mathcal E}$, the function \,$\varphi \circ f$ \,belongs to \,$\mathcal F$, where
$$
{\mathcal F} := \{f \in {\mathcal P}: \, f \hbox{ is nonconstant on any interval of } \, \R\}.
$$
By the chain rule, \,$\varphi\circ f$ \,is a differentiable function and
\,$(\varphi \circ f)'(x) = \varphi ' (f(x)) \, f'(x)$ $(x \in \R )$. Hence \,$(\varphi \circ f)'$ \,vanishes at least on \,$D$, so this derivative vanishes on a dense set.
It remains to prove that \,$\varphi\circ f$ \, is nonconstant on any open interval of \,$\R$.

\vskip .15cm

In order to see this, fix one such interval \,$J$. Clearly, the function \,$\varphi '$ \,also belongs to \,$\mathcal E$. Then \,$\varphi '$ \,is a nonzero entire function. Therefore the set $$S := \{x \in \R : \, \varphi ' (x) = 0\}$$
is discrete in \,$\R$. In particular, it is closed in \,$\R$ \,and countable, so \,$\R \setminus S$ \,is open and dense in \,$\R$. Of course, $S \cap (0,1)$ \,is discrete in \,$(0,1)$. Since \,$f:\R \to (0,1)$ \,is a homeomorphism, the set \,$f^{-1}(S)$ \,is discrete in \,$\R$. Hence
\,$J \setminus f^{-1}(S)$ \, is a nonempty open set of \,$J$.
On the other hand, since \,$D$ \, is dense in \,$\R$, it follows that the set \,$D^0$ of all interior points of \,$D$ \, is \,$\varnothing$. Indeed, if this were not true, there would exist an interval \,$(c,d) \subset D$. Then \,$f' = 0$ \,on \,$(c,d)$, so \,$f$ \,would be constant on \,$(c,d)$, which is not possible because \,$f$ \,is strictly increasing.
Therefore \,$\R \setminus D$ \,is dense in \,$\R$, from which one derives that \,$J \setminus D$ \,is dense in \,$J$.
Thus \,$(J \setminus f^{-1}(S)) \cap (J \setminus D) \neq \varnothing$. Finally, pick any point \,$x_0$ \,in the last set. This means that \,$x_0 \in J$, $f(x_0) \not\in S$
\,(so \,$\varphi ' (f(x_0)) \ne 0$) \,and \,$x_0 \notin D$ \,(so \,$f'(x_0) \ne 0$). Thus
$$
(\varphi\circ f)'(x_0) = \varphi'(f(x_0)) f'(x_0) \neq 0,
$$
which implies that \,$\varphi\circ f$ \, is nonconstant on \,$J$, as required.
\end{proof}

\vskip .1cm

\begin{remarks}
{\rm
%
\noindent 1. In view of the last theorem one might believe that the expression ``$f' = 0$ \,on a dense set'' (see the definition of \,$\mathcal P$) could be replaced by the stronger one
``$f' = 0$ \,almost everywhere''. But this is not possible because every differentiable function is an N-function --that is, it sends sets of null measure into sets of null measure--
(see \cite[Theorem 21.9]{vanrooijschikhof1982}) and every continuous N-function on an interval whose derivative vanishes almost everywhere must be a constant
(see \cite[Theorem 21.10]{vanrooijschikhof1982}).

\vskip .9pt

\noindent 2. If a real function \,$f$ \,is a derivative then \,$f^2$ \,may be not a derivative (see \cite[p.~86]{vanrooijschikhof1982}). This leads us to conjecture that the set
\,$\mathcal{DP}$ of Pompeiu derivatives (and of course, any subset of it) is not algebrable.

\vskip .9pt

\noindent 3. Nevertheless, from Theorem 3.6 (and also from Theorem 4.1) of \cite{gamezmunozsanchezseoane2010} it follows
that the family \,$\mathcal{BDP}$ \,of bounded Pompeiu derivatives is $\mathfrak c$-lineable. A quicker way to see this is by invoking the fact that
\,$\mathcal{BDP}$ \,is a vector space that becomes a Banach space under the supremum norm \cite[pp.~33--34]{bruckner1978}.
Since it is not finite dimensional, a simple application of Baire's category theorem yields \,dim$\,(\mathcal{BDP}) = \mathfrak{c}$. Now, on one hand, we have that, trivially,
\,$\mathcal{BDP}$ \,is dense-lineable in itself. On the other hand, it is known that the set of derivatives that are
positive on a dense set and negative on another is a dense $G_\delta$ set in the Banach space \,$\mathcal{BDP}$ \cite[p.~34]{bruckner1978}. Then, as the authors of  \cite{gamezmunozsanchezseoane2010} suggest, it would be interesting to see whether this set is also dense-lineable.
}
\end{remarks}

\section{Discontinuous functions}


\quad Let \,$n \ge 2$ \,and consider the function \,$f : \R^n \to \R$ \, given by
\begin{equation}\label{desc-func-cont-separately}
	f(x_1,\ldots,x_n)=
	\left\{
	\begin{array}{lll}
		\dfrac{x_1 \cdots x_n}{x_1^{2n} + \cdots + x_n^{2n}} & \text{if} & x_1^2 + \cdots + x_n^2 \neq 0, \\
		0 & \text{if} & x_1 = \cdots = x_n = 0.
	\end{array}
	\right.
\end{equation}
Observe that \,$f$ \,is discontinuous at the origin since arbitrarily near of \,$0 \in \R^n$ \,there exist points of the form \,$x_1 = \cdots = x_n=t$ \,at which \,$f$ \, has the value \,${1 \over nt^n}$. On the other hand, fixed \,$(x_1,\ldots,x_{i-1},x_{i+1},\ldots,x_n)\in\R^{n-1}$, the real-valued function of a real variable given by \,$\psi: x_i \mapsto f(x_1,\ldots,x_n)$ \,is everywhere a continuous function of \,$x_i$. Indeed, this is trivial if all \,$x_j$'s $(j \ne i)$ \,are not \,$0$, while \,$\psi \equiv 0$ \,if some \,$x_j = 0$. Of course, $f$ \,is continuous at any point of \,$\R^n \setminus \{0\}$.

\vskip .15cm

Given \,$x_0 \in \R^n$, we denote by \,$\mathcal{SC}(\R^n,x_0)$ \,the vector space of all {\it separately continuous} functions \,$\R^n \to \R$ \,that are {\it continuous on} \,$\R^n \setminus \{x_0\}$. Since \,card$\,({\mathcal C} (\R^n \setminus \{x_0\})) = \mathfrak{c}$, it is easy to see that
the cardinality (so the dimension) of \,$\mathcal{SC}(\R^n,x_0)$ \,equals \,$\mathfrak{c}$.
Theorem \ref{Thm-DSC(n)-c-algebrable} below will show the algebrability of the family
$$\mathcal{DSC} (\R^n,x_0):= \{f \in \mathcal{SC}(\R^n,x_0): \,f \hbox{ is discontinuous at } x_0\}$$
in a maximal sense.

\begin{theorem}\label{Thm-DSC(n)-c-algebrable}
Let \,$n \in \N$ \,with \,$n \ge 2$, and let \,$x_0 \in \R^n$. Then the set \,$\mathcal{DSC} (\R^n,x_0)$ \,is strongly $\mathfrak{c}$-algebrable.
\end{theorem}

\begin{proof}
We can suppose without loss of generality that \,$x_0 = 0 = (0,0, \dots ,0)$. Consider the function \,$f \in \mathcal{DSC} (\R^n,0)$ \,given by \eqref{desc-func-cont-separately}. For each \,$c > 0$, we set
$$
\varphi_c (x) := e^{|x|^c} - e^{-|x|^c}.
$$
It is easy to see that these functions generate a free algebra. Indeed, if \,$P(t_1, \dots ,t_p)$ \,is a nonzero polynomial in \,$p$ \,variables with \,$P(0,0, \dots ,0) = 0$ \,and \,$c_1, \dots , c_p$ \,are distinct positive real numbers, let \,$M := \{j \in \{1, \dots ,p\}:$ the variable \,$t_j$ \,appears explicitly in the expression of $P\}$, and \,$c_0 := \max \{c_j: \,j \in M\}$. Then one derives that the function
\,$P(\varphi_{c_1}, \dots , \varphi_{c_p})$ \,has the form \,$D \, e^{m|x|^{c_0} + g(x)} + h(x)$, where \,$D \in \R \setminus \{0\}$, $m \in \N$, $g$ \,is a finite sum of the form \,$\sum_k m_k |x|^{\al_k}$ \,with \,$m_k$ \,integers and \,$\al_k < c_0$, and \,$h$ \,is a finite linear combination of functions of the form \,$e^{q(x)}$ \,where, in turn, each \,$q(x)$ \,is a finite sum of the form \,$\sum_k n_k |x|^{\gamma_k}$, with each $\gamma_k$ satisfying that either $\gamma_k < c_0$, or $\gamma_k = c_0$ and $n_k < 0$ simultaneously. Then
\begin{equation}\label{limitofP}
\lim_{x \to \infty} |P(\varphi_{c_1}(x), \dots , \varphi_{c_p}(x))| = +\infty
\end{equation}
and, in particular, $P(\varphi_{c_1}, \dots , \varphi_{c_p})$ \,is not \,$0$ \,identically. This shows that the algebra \,$\Lambda$ \,generated by the \,$\varphi_c$'s \,is free.

\vskip .15cm

Now, define the set \,$\mathcal A$ \,as
$$
{\mathcal A} = \{\varphi \circ f : \, \varphi \in \Lambda \}.
$$
Plainly, $\mathcal A$ \,is an algebra of functions \,$\R^n \to \R$ \,each of them being continuous on \,$\R^n \setminus \{0\}$. But, in addition, this algebra is freely generated by the functions \,$\varphi_c \circ f$ $(c > 0)$. To see this, assume that
$$
\Phi = P(\varphi_{c_1} \circ f, \dots , \varphi_{c_p} \circ f) \in {\mathcal A},
$$
where $P, c_1, \dots ,c_p$ are as above. Suppose that $\Phi = 0$. Evidently, the function \,$f$ \,is onto (note that, for example, $f(x,x, \dots ,x) = {1 \over n \, x^n}$, $f(-x,x,x, \dots ,x)$ $= -{1 \over n \, x^n}$ \,and
$f(0, \dots ,0) = 0$). Therefore $P(\varphi_{c_1}(x), \dots , \varphi_{c_p}(x)) = 0$ \,for all \,$x \in \R$, so $P \equiv 0$, which is absurd because $P(\varphi_{c_1}, \dots ,\varphi_{c_p})$ becomes large as $x \to \infty$.

\vskip .15cm

Hence our only task is to prove that every function \,$\Phi \in \mathcal{A} \setminus \{0\}$ as in the last paragraph belongs to \,$\mathcal{DSC} (\R^n,0)$. Firstly, the continuity of each \,$\varphi_c$ \,implies that \,$\Phi \in \mathcal{SC}(\R^n,0)$. Finally, the function \,$\Phi$ \,is discontinuous at the origin. 
Indeed, we have for all \,$x \ne 0$ \,that
$$
|\Phi (x,x, \dots ,x)| = \left|P \big(\varphi_{c_1}({1 \over n \, x^n}), \dots ,\varphi_{c_p}({1 \over n \, x^n}) \big) \right| \longrightarrow +\infty
$$
as \,$x \to 0$, due to \eqref{limitofP}. This is inconsistent with continuity at \,$0$. The proof is finished.
\end{proof}

\section{Series for which the ratio test or the root test fail}

\quad Every real sequence \,$(a_n)$ \,generates a real series \,$\sum_n a_n$. In order to make the notation of this section consistent, we adopt the convention \,$a / 0 := \infty$ \,for every real number \,$a \in (0,\infty )$, and \,$0/0 := 0$. And a series \,$\sum_n a_n$ \,will be called divergent just whenever it does not converge.
As it is commonly known, given a series \,$\displaystyle \sum_n a_n$, a refinement of the classical {\it ratio test} states that
\begin{itemize}
	\item[(i)] if \,$\displaystyle \limsup_{n \to \infty} \frac{|a_{n+1}|}{|a_n|} < 1$ \,then \,$\displaystyle \sum_n a_n$ \,converges, and
	\item[(ii)] if \,$\displaystyle \liminf_{n \to \infty} \frac{|a_{n+1}|}{|a_n|} > 1$ \,then \,$\displaystyle \sum_n a_n$ \,diverges.
\end{itemize}
However, we can have convergent (positive) series for which $\displaystyle \limsup_{n \to \infty} \frac{a_{n+1}}{a_n} > 1$ and  $\displaystyle \liminf_{n \to \infty} \frac{a_{n+1}}{a_n} < 1$ simultaneously. For instance, consider the series
$$
\sum_n 2^{-n + (-1)^n} = \frac{1}{2^2} + \frac{1}{2} + \frac{1}{2^4} + \frac{1}{2^3} + \ldots = 1.
$$
Making \,$a_n = 2^{-n + (-1)^n}$, we have
$$
\limsup_{n \to \infty} \frac{a_{n+1}}{a_n}=2  \qquad \text{ and } \qquad \liminf_{n \to \infty} \frac{a_{n+1}}{a_n} = {1 \over 8}.
$$
Now, the series \,$\sum_n b_n = \sum_n 2^{n + (-1)^n}$ \,diverges with the same corresponding limsup and liminf.

\vskip .15cm

Analogously, a refinement of the classical {\it root test} asserts that
\begin{itemize}
	\item[(i)] if \,$\displaystyle \limsup_{n \to \infty} |a_n|^{1/n} < 1$ \,then \,$\displaystyle \sum_n a_n$ \,converges, and
	\item[(ii)] if \,$\displaystyle \limsup_{n \to\infty} |a_n|^{1/n} > 1$ \,then \,$\displaystyle \sum_n a_n$ \,diverges.
\end{itemize}
But no of these conditions is sufficient because, for instance, the positive series \,$\sum_n 1/n$ \,converges, the series \,$\sum_n n$ \,diverges but \,$\limsup_{n \to \infty} a_n^{1/n} = 1$ \,for both of them.

\vskip .15cm

Our goal in this section is to show that the set of convergent series for which the ratio test or the root test fails --that is, the refinements of both tests provide no information whatsoever-- is lineable in a rather strong sense; see Theorem \ref{Thm-test-series-lineability} below. The same result will be shown to happen for divergent series.

\vskip .15cm

In order to put these properties into an appropriate context, we are going to consider the space \,$\omega := \R^{\N}$ \,of all real sequences and its subset \,$\ell_1$, the space of all absolutely summable real sequences. Recall that \,$\omega$ \,becomes a Fr\'echet space under the product topology, while \,$\ell_1$ \,becomes a Banach space (so a Fr\'echet space as well) if it is endowed with the $1$-norm \,$\|(a_n)\|_1 := \sum_{n \ge 1} |x_n|$. Moreover, the set \,$c_{00} := \{(a_n) \in \omega : \, \exists n_0 = n_0((a_n)) \in \N$ \,such that \,$a_n = 0 \,\,\, \forall n > n_0\}$ \,is a dense vector subspace of both \,$\omega$ \,and \,$\ell_1$.
A standard application of Baire's category theorem together with the separability of these spaces yields that their dimension equals \,$\mathfrak{c}$.

\vskip .15cm

We need an auxiliary, general result about lineability. Let \,$X$ \,be a vector space and \,$A, \, B$ \,be two subsets of \,$X$. According to \cite{arongarciaperezseoane2009}, we say that \,{\it $A$ \,is stronger than \,$B$} \,whenever \,$A + B \subset A$. The following assertion --of which many variants have been proved-- can be found in \cite{aronbernalpellegrinoseoane2015,bernalpellegrinoseoane2014} and the references contained in them.

\begin{lemma} \label{maxdenslineable-criterium}
Assume that $X$ is a metrizable topological vector space. Let \,$A \subset X$ \,be a maximal lineable. Suppose that there exists a dense-lineable subset \,$B \subset X$ such that \,$A$ \,is stronger than \,$B$ \,and \,$A \cap B = \varnothing$. Then \,$A$ \,is maximal dense-lineable in \,$X$.
\end{lemma}

\begin{theorem} \label{Thm-test-series-lineability}
The following four sets are maximal dense-lineable in \,$\ell_1$, $\ell_1$, $\omega$ \,and \,$\omega$, respectively:
\begin{enumerate}
\item [\rm (a)] The set of sequences in \,$\ell_1$ \,for whose generated series the ratio test fails.
\item [\rm (b)] The set of sequences in \,$\ell_1$ \,for whose generated series the root test fails.
\item [\rm (c)] The set of sequences in \,$\omega$ \,whose generated series diverges and the ratio test fails.
\item [\rm (d)] The set of sequences in \,$\omega$ \,whose generated series diverges and the root test fails.
\end{enumerate}
\end{theorem}

\begin{proof}
We shall only show the first item, even in a very strong form. Namely, our aim is to prove that the set
$$
{\mathcal A} := \left\{(a_n) \in \ell_1: \limsup_{n \to \infty} \frac{|a_{n+1}|}{|a_n|} = \infty \hbox{ \ and \ }
\liminf_{n \to \infty} \frac{|a_{n+1}|}{|a_n|} = 0\right\}
$$
is maximal dense-lineable. The remaining items can be done in a similar manner and are left to the reader: as a hint, suffice it to say that,
instead of the collection of sequences \,$\{(ns)^{-n+(-1)^n}\}$ \,used for (a), one may use \,$\{n^{-s}\}$, $\{(ns)^{n+(-1)^n}\}$ \,and \,$\{n^{s}\}$, respectively, to prove (b), (c) and (d).

\vskip .15cm

Let us prove (a). Consider, for every real number \,$s>1$, the positive sequence \,$a_{n,s} = (ns)^{-n+(-1)^n}$ \,for \,$n \in \mathbb{N}$.
Since \,$a_{n,s} \le n^{-2}$ \,for all \,$n \ge 3$, the comparison test yields \,$(a_{n,s}) \in \ell_1$. Next, take
$$
E = \text{span}\{(a_{n,s})_{n \ge 1}: \, s>1\},
$$
which is a vector subspace of \,$\ell_1$. It can be easily seen that \,dim($E$)$=\mathfrak{c}$. Indeed, suppose that a linear combination of the type
\begin{equation}\label{linearcombination}
x = (x_n) \hbox{ \ with \ } x_n = \sum_{j=1}^{k} \alpha_j \, a_{n, s_j} \quad (n \ge 1)
\end{equation}
is identically \,$0$. Then, supposing without loss of generality that \,$k \ge 2$ \,and \,$s_1 > s_2 > \cdots > s_k$, and dividing the previous expression by \,$(ns_k)^{-n+(-1)^n}$ \,we obtain
$$
0 = \alpha_1 \left( \frac{s_1}{s_k}\right)^{-n+(-1)^n} + \cdots + \alpha_{k-1} \left( \frac{s_{k-1}}{s_k}\right)^{-n+(-1)^n}   + \alpha_k .
$$
Taking limits in the previous expression, as \,$n$ \,goes to \,$\infty$, we have \,$\alpha_k = 0$. Inductively we can obtain that all \,$\alpha_j$'s \,are \,$0$, having that the set of sequences \,$\{a_{n,s}, s>1\}$ \,is linearly independent, thus \,dim($E$)$ = \mathfrak{c}$.

\vskip .15cm

Next, let us show that, given any sequence \,$x = (x_n)_{n \ge 1} \in E \setminus \{0\}$ \,as in (\ref{linearcombination})
(with \,$\alpha_k \ne 0$ \,and \,$s_1 > \cdots > s_k$),
the ratio test does not provide any information on the convergence of \,$\displaystyle \sum_{n \ge 1} x_n$.
Dividing numerators and denominators by \,$\alpha_k (ns_k)^{-n+(-1)^{n}}$, we get
$$
\begin{array}{rl}
\displaystyle \left|\frac{x_{n+1}}{x_n} \right| =&
\displaystyle \left| \frac{\alpha_1 ((n+1)s_1)^{-n-1+(-1)^{n+1}} +  \cdots + \alpha_k ((n+1)s_k)^{-n-1+(-1)^{n+1}}}{\alpha_1 (ns_1)^{-n+(-1)^{n}} + \cdots + \alpha_k (ns_k)^{-n+(-1)^{n}}}\right| = \\
\medskip
= & \displaystyle \left| \frac{\beta_{1,n} + \cdots + \beta_{k,n}}{\gamma_n + 1}\right|,
\end{array}
$$
where  \,$\gamma_n \to 0$ $(n \to \infty)$, $\beta_{j,n} = {\alpha_j \over \alpha_k} \, \big({n \over n+1}\big)^n \, {s_j^{-2} \, s_k^{-1} \over n(n+1)} \, \big({s_k \over s_j}\big)^n$
\,if \,$n$ \,is even, and \,$\beta_{j,n} = {\alpha_j \over \alpha_k} \, \big({n \over n+1}\big)^n \,n \, s_k \, \big({s_k \over s_j}\big)^n$
\,if \,$n$ \,is odd ($j=1, \dots ,k$). Note that \,$\lim_{n \to \infty \atop n \,\, {\rm even}} \beta_{j,n} = 0$ \,for all \,$j \in \{1, \dots ,k\}$,
$\lim_{n \to \infty \atop n \,\, {\rm odd}} \beta_{j,n} = 0$ \,for all \,$j \in \{1, \dots ,k-1\}$, and \,$\lim_{n \to \infty \atop n \,\, {\rm odd}} |\beta_{k,n}| = \infty$. Then
$$
\liminf_{n \to \infty} \left|\frac{x_{n+1}}{x_n} \right| = 0 \qquad \text{ and } \qquad \limsup_{n \to \infty} \left|\frac{x_{n+1}}{x_n} \right| = \infty .
$$
Consequently, $x$ \,belongs to \,$\mathcal A$, as we wished. This shows that \,$\mathcal A$ \,is maximal lineable in \,$\ell_1$.

\vskip .15cm

Finally, an application of Lemma \ref{maxdenslineable-criterium} with \,$X = \ell_1$, $A = {\mathcal A}$ \,and \,$B = c_{00}$ \,proves the
maximal dense-lineability of \,$\mathcal A$.
\end{proof}

Concerning parts (c) and (d) of the last theorem, one might believe that they happen because root and ratio test are specially non-sharp criteria. To be more precise, given a divergent series \,$\sum_n c_n$ \,with positive terms (notice that we may have $c_n \to 0$, for instance with \,$c_n = 1/n$), one might believe that there are not many sequences \,$(x_n)$ \,essentially lower that \,$(c_n)$ \,such that \,$\sum_n x_n$ \,still diverges. The following theorem will show that this is far from being true. In order to formulate it properly, a piece of notation is again needed. For a given sequence \,$(c_n) \subset (0, \infty )$, we denote by \,$c_0((c_n))$ \,the vector space of all sequences \,$(x_n) \in \omega$ \,satisfying \,$\lim_{n \to \infty} x_n/c_n = 0$. It is a standard exercise to prove that, when endowed with the norm
$$\
\|(x_n)\| = \sup_{n \ge 1} |x_n/c_n|,
$$
the set \,$c_0((c_n))$ \,becomes a separable Banach space, such that \,$c_{00}$ \,is a dense subspace of it.

\begin{theorem} \label{Thmc_0((c_n))maxdenslineable}
Assume that \,$(c_n)$ \,is a sequence of positive real numbers such that the series \,$\sum_n c_n$ \,diverges. Then the family of sequences \,$(x_n) \in c_0((c_n))$ \,such that the series \,$\sum_n x_n$ \,diverges is maximal dense-lineable in \,$c_0((c_n))$.
\end{theorem}

\begin{proof}
By Baire's theorem, dim$(c_0((c_n))) = \mathfrak{c}$. We denote
$$
A := \big\{(x_n) \in c_0((c_n)): \hbox{ the series} \, \sum_n x_n \hbox{ diverges}\big\}.
$$
Obviously, $A + c_{00} \subset A$ \, and \,$A \cap c_{00} = \varnothing$.
Let us apply Lemma \ref{maxdenslineable-criterium} with
$$
X = c_0((c_n)) \hbox{ \ and \ } B = c_{00}.
$$
Then it is enough to show that \,$A$ \,is maximal lineable, that is, $\mathfrak{c}$-lineable.

\vskip .15cm

To this end, we use the divergence of \,$\sum_n c_n$ \,and the fact \,$c_n > 0$ $(n \ge 1)$. Letting \,$n_0 := 1$, we can obtain inductively a sequence \,$\{n_1 < \cdots < n_k < \cdots \} \subset \N$ \,satisfying
$$
c_{n_{k-1}+1} + \cdots + c_{n_k} > k \quad (k = 1,2, \dots ).
$$
Now, define the collection of sequences \,$\{(d_{n,t})_n: \, 0 < t < 1\}$ \,by
$$
d_{j,t} := {c_j \over k^t} \quad (j=n_{k-1}+1, \dots ,n_k; \,k \in \N ).
$$
Since \,$k^t \to \infty$ \,as \,$k \to \infty$, each sequence \,$(d_{n,t})_n$ \,belongs to \,$c_0((c_n))$.
We set
$$M := {\rm span} \, \{(d_{n,t})_n: \, 0 < t < 1\}.$$
This vector space is $\mathfrak{c}$-dimensional. Indeed, if this were not the case, then there would exist \,$s \in \N$, $\la_1, \dots , \la_s \in \R$
\,with \,$\la_1 \ne 0$ \, and \,$0 < t_1 < \cdots < t_s < 1$ \,such that the sequence
$$
\Phi = (x_n)_n = (\la_1 d_{n,t_1} + \cdots + \la_s d_{n,t_2})_n \leqno (6.2)
$$
is identically zero. From the triangle inequality, we obtain for each \,$k \in \N$ \,that
\begin{equation*}
\begin{split}
\left| \sum_{j=n_{k-1}+1}^{n_k} x_j \right| &= \left| \sum_{j=n_{k-1}+1}^{n_k} \sum_{\nu = 1}^s \la_{\nu} d_{j,t_\nu} \right| \\
                                            &\ge \left| |\la_1| \sum_{j=n_{k-1}+1}^{n_k} d_{j,t_1} - \sum_{\nu = 2}^s |\la_\nu | \sum_{j=n_{k-1}+1}^{n_k} d_{j,t_\nu} \right|\\
                                            & =\left(\sum_{j=n_{k-1}+1}^{n_k} c_j \right) \left| {|\la_1| \over k^{t_1}} - \sum_{\nu = 2}^s   {|\la_\nu| \over
                                               k^{t_\nu}}  \right| \\
&> \left| |\la_1| \, k^{1 - t_1} - \sum_{\nu = 2}^s |\la_\nu| \, k^{1 - t_\nu} \right| \longrightarrow \infty \hbox{ \ as } k \to \infty ,
\end{split}
\end{equation*}
which is absurd. Hence, the sequences \,$(d_{n,t})$ $(0<t<1)$ \,are linear independent and \,dim$(M) = \mathfrak{c}$. Finally, we prove that each
\,$\Phi = (x_n) \in M \setminus \{0\}$ \,belongs to \,$A$. Note that such a sequence \,$\Phi$ \,has the shape given in (6.2), with \,$\la_1 \ne 0$ \, and \,$0 < t_1 < \cdots < t_s < 1$. But the fact \,$| \sum_{j=n_{k-1}+1}^{n_k} x_j | \to \infty$ $(k \to \infty )$ \,as shown above entails that the Cauchy convergence criterium for series does not hold for \,$\sum_n x_n$. Consequently, this series diverges, as required.
\end{proof}

\begin{remark}
{\rm The property given in Theorem \ref{Thmc_0((c_n))maxdenslineable} is topologically generic too, that is, the set \,$A$ \,above is {\it residual} in
\,$c_0((c_n))$. Indeed, we have that \,$A \supset \bigcap_{M \in \N} \bigcap_{N \in \N} \bigcup_{m > n > N} A_{M,N}$, where
$$
A_{M,N} := \big\{(x_n) \in c_0((c_n)): \, |x_n + x_{n+1} + \cdots + x_m| > M\big\},
$$
and each set \,$B_{M,N} := \bigcup_{m > n > N} A_{M,N}$ \,is open and dense in \,$c_0((c_n))$. To prove this, fix \,$M,N,m,n \in \N$ \,with \,$m > n > N$\ and observe  that \,$A_{M,N} = c_0((c_n)) \cap \Psi^{-1} ((-\infty ,-M) \cup (M, \infty ))$, where \,$\Psi : \omega \to \R$ \,is given by \,$\Psi ((x_j)) = x_m + \cdots + x_n$. The continuity of the projections \,$(x_j) \in \omega \mapsto x_k \in \R$ $(k \in \N )$ \,entails the continuity of \,$\Psi$, so \,$\Psi^{-1} ((-\infty ,-M) \cup (M, \infty ))$ \,is open in \,$\omega$. The inclusion \,$c_0((c_n)) \hookrightarrow \omega$ \,being continuous, we get that \,$A_{M,N}$ \,is open in \,$c_0((c_n))$. Therefore \,$B_{M,N}$ \,is also open in \,$c_0((c_n))$. As for the density of \,$B_{M,N}$, note that, due to the density of \,$c_{00}$ \,in \,$c_0((c_n))$, it is enough to show that, given \,$\Phi = (y_j) = (y_1,y_2, \dots ,y_s,0,0,0, \dots ) \in c_{00}$ \,and \,$\ve > 0$, there exist
\,$m > n > N$ \,and \,$(x_j) \in c_0((c_n))$ \,with \,$|x_m + \cdots + x_n| > M$ \,and \,$\|(x_j) - \Phi \| < \ve$. From the divergence of the positive series \,$\sum_j c_j$, it follows the existence of \,$m,n \in \N$ \,such that \,$m > n > \max \{s,N\}$ \,and \,$c_n + \cdots + c_m > 2M/\ve$.
Define the sequence \,$(x_j) \in c_{00} \subset c_0((c_n))$ \,by
$$
x_j =
	\left\{
	\begin{array}{ll}
		y_j & \text{if \ }  j \in \{1, \dots ,s\}, \\
		{\ve \, c_j \over 2} & \text{if \ } j \in \{m, \dots ,n\}, \\
        0 & \text{otherwise.}
	\end{array}
	\right.
$$
By construction, $|x_m + \cdots + x_n| > M$. Finally, $\|(x_j) - \Phi \| = \sup_{j \ge 1} {|x_j - y_j| \over c_j} =  \sup_{n \le j \le m} {|x_j| \over c_j} =   {\ve \over 2} < \ve$, as required.}
\end{remark}

\section{Convergence in measure versus convergence almost everywhere}

\quad Let \,$m$ \,be the Lebesgue measure on \,$\R$. In this section we will restrict ourselves to the interval \,$[0,1]$, which of course has finite measure \,$m([0,1]) = 1$. Denote by \,$L_0$ \,the vector space of all Lebesgue measurable functions \,$[0,1] \to \R$, where two functions are identified whenever they are equal almost everywhere (a.e.) in \,$[0,1]$. Two natural kinds of convergence of functions of \,$L_0$ \,are a.e.-convergence and convergence in measure. Recall that a sequence \,$(f_n)$ \,of measurable functions is said to converge in measure to a measurable function \,$f: [0,1] \to \R$ \,provided that
$$
\lim_{n \to \infty} m(\{x \in [0,1] : \, |f_n(x) - f(x)| > \alpha \}) = 0 \hbox{ \ for all \ } \alpha > 0.
$$
Convergence in measure is specially pleasant because it can be described by a natural metric on \,$L_0$; see e.g.~\cite{nielsen1997}. Namely, the distance
$$
\rho (f,g) = \int_{[0,1]} {|f(x) - g(x)| \over 1 + |f(x) - g(x)|} \,dx \quad (f,g \in L_0)
$$
satisfies that \,$f_n  \mathop{\longrightarrow}\limits^{\rho}_{n \to \infty} f$ \,if and only if \,$f_n \mathop{\longrightarrow}\limits_{n \to \infty} f$ \,in measure (the finiteness of the measure of $[0,1]$ is crucial).
Under the topology generated by \,$\rho$, the space \,$L_0$ \,becomes a complete metrizable topological vector space for which the set \,${\mathcal S}$ \,of
simple (i.e.~of finite image) measurable functions forms a dense vector subspace. Actually, $L_0$ \,is separable because the set \,$S_0 \,\, (\subset S)$ \,of finite linear combinations with rational coefficients of functions of the form \,$\chi_{[p,q]}$ ($0 \le p < q \le 1$ rational numbers) is countable and dense in \,$L_0$. Here \,$\chi_A$ \,denotes the indicator function of the set \,$A$.

\vskip .15cm

Convergence in measure of a sequence \,$(f_n)$ \,to \,$f$ \,implies a.e.-convergence to \,$f$ \,of some subsequence \,$(f_{n_k})$ (see \cite[Theorem 21.9]{nielsen1997}). But, generally, this convergence cannot be obtained for the whole sequence \,$(f_n)$. For instance, the so-called ``typewriter sequence'' given by
$$T_n = \chi_{[j2^{-k},(j+1)2^{-k}]}$$
(where, for each $n$, the non-negative integers $j$ and $k$ are uniquely determined by $n=2^k +j$ and $0 \le j < 2^k$) satisfies that \,$T_n \to f \equiv 0$ \,in measure but, for every point $x_0 \in [0,1]$, the sequence \,$\{T_n(x_0)\}_{n \ge 1}$ \,does not converge.

\vskip .15cm

In order to face the lineability of this phenomenon, we need, once more, to put the problem in an adequate framework. Let \,$L_0^{\N}$ \,be the space of all sequences of measurable functions \,$[0,1] \to \R$, endowed with the product topology. Since \,$L_0$ \,is metrizable and separable, the space \,$L_0^{\N}$ \,is also a complete metrizable separable topological vector space. Again, by Baire's theorem, this implies \,dim$(L_0^{\N}) = \mathfrak{c}$.  Moreover, the set
$$
\big\{\Phi = (f_n) \in L_0^{\N}: \, \exists N = N(\Phi ) \in \N \hbox{ such that } f_n = 0 \hbox{ for all } n \ge N\big\}  \leqno (7.1)
$$
is dense in the product space. Now, we are ready to state our next theorem, with which we finish this paper.

\begin{theorem}
The family of Lebesgue classes of sequences \,$(f_n) \in L_0^{\N}$ \,such that \,$f_n \to 0$ \,in measure but \,$(f_n)$ \,does not converge almost everywhere in \,$[0,1]$ \,is maximal dense-lineable in \,$L_0^{\N}$.
\end{theorem}

\begin{proof}
Let \,$(T_n)$ \,be the typewriter sequence defined above, and let \,$A$ \,be the family described in the statement of the theorem, so that \,$(T_n) \in A$.
Extend each \,$T_n$ \,to the whole \,$\R$ \,by defining \,$T_n(x) = 0$ \,for all \,$x \notin [0,1]$.
It is readily seen that, for each \,$t \in (0,1/2)$, the translated-dilated sequence \,$T_{n,t} (x) := T_n(2(x-t))$ $(n \ge 1)$ \,also tends to \,$0$ \,in measure. Consider the vector space
$$
M := \hbox{span} \, \{(T_{n,t}): \, t > 0\}.
$$
The sequences \,$(T_{n,t})$ $(0 < t < 1/2)$ \,are linearly independent. Indeed, if it were not the case, there would be \,$0 < t_1 < t_2 < \cdots < t_s < 1/2$ \,as well as real numbers \,$c_1, \dots ,c_s$ \,with \,$c_s \ne 0$ \,such that \,$c_1 T_{n,t_1} + \cdots + c_s T_{n,t_s} = 0$ \,for all \,$n \in \N$. In particular, $c_1 T_{1,t_1}(x) + \cdots + c_s T_{1,t_s}(x) = 0$ \,for almost all \,$x \in \R$. But \,$T_1 = \chi_{[0,1]}$,
so \,$T_{1,t} = \chi_{[t,t+{1\over2}]}$ \,for all \,$t > 0$. Therefore
$$
c_1 \chi_{[t_1,t_1+{1\over2}]}(x) + \cdots + c_s \chi_{[t_s,t_s+{1\over2}]}(x) = 0 \hbox{ \ for almost all \,} x \in [0,1].
$$
But, for every \,$x \in (\max \{t_{s-1}+{1\over2},t_s\},t_s + {1\over2}]$, the left-hand side of the last expression equals \,$0 + \cdots + 0 + c_s \cdot 1 = c_s \ne 0$, which is absurd. This shows the required linear independence. Then \,${\rm dim} \,(M) = \mathfrak{c}$. Moreover, since \,$L_0$ \,is a topological vector space carrying the topology of convergence in measure, we get that every member of \,$(F_n) := (c_1 T_{n,t_1} + \cdots + c_s T_{n,t_s}) \in M$ \,is a sequence tending to \,$0$ \,in measure.

\vskip .15cm

Next, fix any \,$(F_n) \in M$ \,as above, with \,$0 < t_1 < t_2 < \cdots < t_s$ \,and \,$c_s \ne 0$.
For all \,$x \in (\max \{t_s,t_{s-1}+{1\over2}\},t_s+{1\over2}]$, we have
$$
F_n(x) = \sum_{j=1}^s c_j \, T_{n,t_j} (x) =  \sum_{j=1}^s c_j \, T_{n} (2(x - t_j)) = c_s \, T_n (2(x - t_s)).
$$
Since \,$c_s \ne 0$ \,and \,$(T_n(y))$ \,does not converge for every \,$y \in (\max \{0,2(t_{s-1} - t_s) + 1\},1] \,\, (\subset [0,1])$, we derive that, for each
\,$x \in (\max \{t_s,t_{s-1}+{1\over2}\},t_s+{1\over2}]$, the sequence \,$(F_n(x))$ \,does not converge. This shows that \,$M \setminus \{0\} \subset A$. Thus, $A$ \,is $\mathfrak{c}$-lineable. Finally, an application of Lemma \ref{maxdenslineable-criterium} with
$$
X = L_0^{\N} \hbox{ \ and \ } B = \hbox{the set given by (7.1)}
$$
puts an end on the proof.
\end{proof}


\end{document}